\documentclass[a4paper,11pt]{amsart}

\usepackage{latexsym}
\usepackage{amssymb,amsmath}
\usepackage{graphics}
\usepackage{url}
\usepackage[vcentermath]{youngtab}

\newtheorem{theorem}{Theorem}

\newtheorem{proposition}[theorem]{Proposition}

\newtheorem{lemma}[theorem]{Lemma}

\theoremstyle{definition}
\newtheorem{definition}[theorem]{Definition}
\newtheorem{remark}[theorem]{Remark}

\newcommand{\N}{\mathbb{N}}
\newcommand{\Z}{\mathbb{Z}}
\newcommand{\Q}{\mathbb{Q}}

\newcommand{\F}{\mathbb{F}}

\renewcommand{\epsilon}{\varepsilon}

\DeclareMathOperator{\GL}{GL}

\DeclareMathOperator{\End}{End}

\DeclareMathOperator{\Tr}{Tr}

\newcommand{\Ind}{\big\uparrow}
\newcommand{\Res}{\big\downarrow}

\newcommand\blfootnote[1]{%
  \begingroup
  \renewcommand\thefootnote{}\footnote{#1}%
  \addtocounter{footnote}{-1}%
  \endgroup
}

\newcounter{thmlistcnt}
	{\setcounter{thmlistcnt}{0}%
	\begin{list}{\emph{(\roman{thmlistcnt})}}{%
		\usecounter{thmlistcnt}%
		\setlength{\topsep}{0pt}%
		\setlength{\leftmargin}{0pt}%
		\setlength{\itemsep}{0pt}%
		\setlength{\itemindent}{17pt}}%
	}%
	{\end{list}}%

\begin{document}
\title[Nakayama's Conjecture]{A New Proof of Nakayama's Conjecture via Brauer Quotients of Young Modules}
\date{\today}
\author{William O'Donovan}
\begin{abstract}
We provide a self-contained proof of the main properties of Brauer quotients of Young modules. We then use these results to give a new inductive proof of Nakayama's Conjecture on the blocks of the symmetric group.
\end{abstract}
\maketitle
\thispagestyle{empty}

\section{Introduction}

Let $G$ be \blfootnote{2010 \emph{Mathematics Subject Classification:} 20C05, 20C30} a \blfootnote{\emph{Key words and phrases.} symmetric group, block, Young modules, Brauer quotient, vertex} finite group and $p$ be a prime. We say that a \emph{block} (or a $p$-block) of $G$ is a primitive idempotent $e$ in the centre of the group algebra $\F_p G$. We say that an indecomposable $\F G$-module $U$ \emph{lies in the block} $e$ (or belongs to $e$) if $eU = U$. Understanding the blocks of a group is an important problem in modular representation theory: for example, sorting the simple and indecomposable modules of $\F G$ into blocks yields a block diagonal decomposition of the Cartan matrix of $G$, as described in \cite[Corollary 12.1.8]{Webb} (hence the nomenclature). 

Unfortunately, in general it is also a difficult problem to understand the blocks of a finite group. An exception, however, is in the case of the symmetric group, where there is a beautiful combinatorial characterisation of the blocks, given by a result still known as Nakayama's Conjecture.  

We define a \emph{partition} of $n \in \N_0$ to be a sequence $ \lambda = (\lambda_1,\lambda_2,\ldots)$ of non-increasing non-negative integers, such that $\sum_i \lambda_i = n$. There is a close connection between much of the representation theory of the symmetric group $S_n$ and the combinatorics of partitions. A partition may be visualised by means of its \emph{Young diagram}, which is an array consisting of $\lambda_1$ boxes in the first row, $\lambda_2$ boxes in the second row, and so on.

A node $(i,j)$ in the Young diagram $[\lambda]$ of $\lambda$ is said to form part of the \emph{rim} if $(i+1,j+1) \notin [\lambda]$. A collection of $p$ edge-connected nodes in the rim of $[\lambda]$ is a \emph{$p$-hook} if their removal from $[\lambda]$ leaves the Young diagram of a partition. We define the \emph{$p$-core} of $\lambda$, which we denote by $c_p(\lambda)$, to be the partition obtained by repeatedly removing all $p$-hooks from $\lambda$. The number of $p$-hooks removed is called the \emph{$p$-weight} of $\lambda$. It is fairly easy to see using the abacus notation for partitions (see \cite[p.76--78]{JamKer}) that the $p$-core of a partition is independent of the manner in which we remove the $p$-hooks, and accordingly is well-defined. 

The importance of $p$-cores is that they label the blocks of the symmetric group, in the following sense, as Nakayama conjectured in 1940:

\begin{theorem}The blocks of the symmetric group are labelled by pairs $(\gamma,w)$, where $\gamma$ is a $p$-core and $w \in \N_0$ is such that $n = |\gamma| + pw$. Thus the Specht module $S^\lambda$ lies in the block $(\gamma, w)$ of $S_n$ if and only if $c_p(\lambda) = \gamma$.\end{theorem} 

Nakayama's Conjecture was proved by Brauer and Robinson in 1947; see \cite{Brauer} and \cite{Robinson}. Since then, many proofs have been found. A proof using Brauer pairs can be found in \cite{Br2}. Murphy gave a proof of Nakayama's Conjecture by explicitly constructing a complete set of primitive idempotents for the symmetric group algebra in prime characteristic using Murphy operators in \cite{Mur}. Perhaps the shortest proof is \cite{MeiTap}, which uses generalised decomposition numbers; this argument can also be found in English in \cite[p.270--275]{JamKer}. 

With such an abundance of proofs, picking a favourite is a matter of taste. In this paper, we give a new proof of Nakayama's Conjecture, which we feel has two points to recommend it. First, our proof will use only the representation theory of the symmetric group, in keeping with the philosophy that results about the symmetric group deserve to be proved with just the machinery of the symmetric group. Secondly, our proof is comparatively elementary: it is free from any kind of calculation, and no more than basic knowledge of block theory is required. 

Our main tool will be the Brauer quotients of Young modules; the fundamental results on these were first proved by Grabmeier and Klyachko in \cite{Grab} and \cite{Kly}, using the Schur algebra. A proof using only the representation theory of the symmetric group was published in \cite{Erd}, and corrected in the setting of the general linear group in \cite{ErdSch}. Although the proof given in \cite{ErdSch} can be streamlined considerably for symmetric groups, this has not appeared in print. Moreover, some of the intermediate results used to obtain this proof will be required in our proof of Nakayama's Conjecture. For these reasons, we provide an account of the Brauer quotients of Young modules and a self-contained proof of the key result on these modules, stated as Theorem \ref{main} below.

The paper is structured as follows. In Section 2, we set out the important background on the Brauer quotient. We then re-prove the key properties of the Brauer quotients of Young modules. Section 4 is a summary of the results from block theory which we shall need throughout the proof of Nakayama's Conjecture. We begin proving Nakayama's Conjecture in Section 5 by induction on the degree of the symmetric group; step 1 is an easy base case. Next we show that Young modules having a common vertex lie in the same block if and only if their labelling partitions have the same p-core, in Step 2. In Step 3, we find a way to compare the blocks of non-projective Young modules which have different vertices. We conclude in step 4 by understanding the blocks of projective Young modules (about which taking Brauer quotients provides no information). We deal with Young modules labelled by $p$-core partitions by showing that they are simple and projective; for other $p$-restricted partitions we have to use the Mullineux map and duality to understand their blocks.

\section{The Brauer Quotient}

Throughout, let $p$ be a prime number, $\F$ be a field of characteristic $p$ and $G$ be a finite group.  For a fuller exposition of the material in this section, see Brou\'e's original paper \cite{Br1}. 

We say that an $\F G$-module $V$ is a \emph{$p$-permutation module} if whenever $P$ is a $p$-subgroup of $G$, there is a linear basis of $V$ which is permuted by $P$. It is not too hard to see that the $p$-permutation modules are precisely the $\F G$-modules with trivial source (see \cite[0.4]{Br1}).  

Given a $p$-subgroup $Q$ of $G$, define $V^Q = \lbrace v \in V : \text{$qv=v$ for all $q \in Q$} \rbrace$, the set of $Q$-fixed elements of $V$. Let $R$ be a subgroup of $Q$, and let $T$ be a transversal for $R$ in $Q$. We define the \emph{trace map} $\Tr^Q_R : V^R \rightarrow V^Q$ by $$\Tr_R^Q(v) = \sum_{g \in T} gv.$$ Now $$ \sum_{R < Q} \Tr_R^Q(V^R)$$ is a $\F N_G(Q)$-module which is contained in $V^Q$; hence we may define the following $\F N_G(Q)$-module: $$V(Q) = V^Q / \sum_{R < Q} \Tr_R^Q(V^R).$$ This is the \emph{Brauer quotient} of $V$ with respect to $Q$.

There is a more tangible way to think of the Brauer quotient. Let $P$ be a Sylow $p$-subgroup of $G$ with corresponding $p$-permutation basis $B$. For each $v \in B$, denote by $P_v$ the stabiliser of $v$ in $P$. The elements $\sum_{g \in P / P_v} gv = \Tr^P_{P_v}(v)$ form a basis of $V^P$. If $P_v$ is a proper subgroup of $P$, then this trace becomes zero on taking the Brauer quotient, so $V(P) = \langle B^P \rangle$.

The importance of the Brauer quotient is that it enables us to determine the vertices of $p$-permutation modules. More precisely, we have:

\begin{theorem}\label{brmain} \cite[Theorem 3.2]{Br1} Let $U$ be a $p$-permutation $\F G$-module and let $R$ be a $p$-subgroup of $G$. Then $R$ is contained in a vertex of $U$ if and only if $U(R) \not= 0$. Moreover, the vertices of $U$ are precisely the maximal $p$-subgroups $P$ of $G$ such that $U(P) \not= 0$. If $U$ has vertex $P$, then $U(P)$ is isomorphic to the Green correspondent of $U$.
\end{theorem}

The following result, known as the Brou\'e correspondence, will also be extremely important.

\begin{theorem}\label{brcorr} \cite[Theorem 3.3]{Br1} Let $P$ be a $p$-subgroup of $G$. The map sending a $\F G$-module $U$ to its Brauer quotient $U(P)$ induces a 1-1 correspondence between isomorphism classes of indecomposable $p$-permutation $\F G$-modules with vertex $P$ and indecomposable projective $N_G(P)/P$-modules.

\end{theorem}

\section{Young Modules for Symmetric Groups}

We begin our account of the theory of Young modules by reminding the reader of the key definitions and notation which we shall use; more details of the representation theory of the symmetric group can be found in \cite{Jam}. 

A sequence of non-negative integers $(\lambda_1,\lambda_2,\ldots)$ such that $\sum_i \lambda_i = n$ is said to be a \emph{composition} of $n \in \N$ (note that in a partition, the sequence must be weakly decreasing). If $\lambda$ is a composition of $n$, we define the corresponding \emph{Young subgroup} $S_\lambda$ to be subgroup of $S_n$ which is the direct product of all the symmetric groups $S_{\lambda_i}$. The \emph{Young permutation module} corresponding to $\lambda$ is the $\F S_n$-module $\F \Ind_{S_\lambda}^{S_n}$ and is denoted by $M^\lambda$. Over a field of characteristic 0, the simple modules for $S_n$ are the \emph{Specht modules}, which are indexed by partitions of $n$; we use the notation $S^\lambda$ for the Specht module labelled by $\lambda$. The character of $S^\lambda$ over a field of characteristic 0 is denoted by $\chi^\lambda$.  

Now let $\F$ be a field of prime characteristic $p$, and let $\lambda$ be a partition of $n$. We say that $\lambda$ is \emph{p-restricted} if $\lambda_i - \lambda_{i+1} < p$ for every $i$, and that $\lambda$ is \emph{p-regular} if no non-zero part of $\lambda$ is repeated $p$ or more times. These sets of partitions give two ways to label the simple $\F S_n$-modules. If $\lambda$ is a $p$-regular partition of $n$, set $D^\lambda := S^\lambda / \text{rad}(S^\lambda)$; if $\lambda$ is a $p$-restricted partition of $n$, we put $D_\lambda := \text{soc}(S^\lambda)$. The sets $\lbrace D^\lambda : \text{$\lambda$ $p$-regular}\rbrace$ and $\lbrace D_\lambda : \text{$\lambda$ $p$-restricted}\rbrace$ form complete sets of nonisomorphic simple $\F S_n$-modules. We denote by $\trianglerighteq$ the dominance order on partitions. 

Write the permutation module $M^\lambda$ as a direct sum of indecomposable $\F S_n$-modules, say $M^\lambda = \bigoplus_i Y_i$. Let $t$ be a $\lambda$-tableau with corresponding signed column sum $\kappa_t$, as defined in \cite[Definitions 4.3]{Jam}. By the Submodule Theorem (see \cite[Theorem 4.8]{Jam}), if $U$ is a submodule of $M^\lambda$ then either $\kappa_t U = 0$ or $S^\lambda \subseteq U$. Since $\kappa_t M^\lambda$ is one-dimensional by \cite[Corollary 4.7]{Jam}, there is a unique summand $Y_j$ such that $\kappa_t Y_j \not=0$. Therefore, $Y_j$ is the unique summand of $M^\lambda$ containing $S^\lambda$ as a submodule; this is called the \emph{Young module} for $\lambda$ and is denoted by $Y^\lambda$. Our goal is to understand the Young modules, which we shall achieve by proving the following result.

\begin{theorem}\label{main}
The Young modules form a complete set of indecomposable, pairwise non-isomorphic, summands of the permutation modules $M^\lambda$. In the decomposition of $M^\lambda$ into indecomposable summands, $Y^\lambda$ appears exactly once, and all other summands are of the form $Y^\mu$, where $\mu \unrhd \lambda$. Write $\lambda = \sum_{i=0}^t \lambda(i)p^i$, with each $\lambda(i)$ a $p$-restricted partition. Let $r_i$ be the degree of $\lambda(i)$, and let $\rho$ be the partition of $n$ which has $r_i$ parts equal to $p^i$. Then a Sylow $p$-subgroup of $S_\rho$, say $P$, is a vertex of $Y^\lambda$, and the Green correspondent of $Y^\lambda$ satisfies the following isomorphism of $\F {N_{S_n}(P)}/{P}$-modules (in the sense as explained before the statement of Lemma \ref{tens}): $$Y^\lambda(P) \cong Y^{\lambda(0)} \otimes \cdots \otimes Y^{\lambda(t)}.$$
\end{theorem}

Suppose that $Y^\lambda$ is a direct summand of $M^\mu$. Then by the above argument, $\kappa_t Y^\lambda \not= 0$, and hence $\kappa_t M^\mu \not= 0$. By \cite[Lemma 4.6]{Jam}, it follows that $\mu \unrhd \lambda$. Furthermore, if $Y^\lambda \cong Y^\mu$, then $Y^\lambda$ is a direct summand of $M^\mu$ and $Y^\mu$ is a summand of $M^\lambda$, whence $\lambda \trianglerighteq \mu$ and $\mu \trianglerighteq \lambda$, and so $\lambda = \mu$.

\begin{remark}It is tempting to argue in the above that, if $t$ is a $\lambda$-tableau, then $\kappa_t S^\lambda \not=0$. Unfortunately, this is false: for example, if $p = n = 2$, $\lambda = (1^2)$ and $t$ is the row-standard $\lambda$-tableau, then it is easy to see that $\kappa_t S^\lambda = 0$. This justifies our taking a slightly longer path than might appear necessary. \end{remark}

We have proved that the Young modules are pairwise non-isomorphic and established our claim that only Young modules labelled by partitions dominating $\lambda$ can appear as summands of $M^\lambda$. To prove the rest of our main result, we shall need to study Brauer quotients of permutation modules. The following lemma is critical; this result was originally given in \cite[Lemma 1]{Erd}, but we provide a new, simpler, proof.

\begin{lemma}\label{goup}Let $M$ be a $p$-permutation $\F G$-module, and let $P$ and $Q$ be $p$-subgroups of $G$ with $Q < P$. Suppose that $M(P) = M(Q)$ as sets. Then $M$ has no summand with vertex $Q$. \end{lemma}
\begin{proof}
Write $M$ as a sum of indecomposable modules, say $M = \bigoplus_{i=1}^n M_i$. For each $i$, let $B_i$ be a $p$-permutation basis of $M_i$ with respect to $P$; observe that $B_i$ is also a $p$-permutation basis with respect to $Q$. Therefore, a basis for $M_i(Q)$ is $B_i^Q$, and a basis for $M_i(P)$ is $B_i^P$. Since $Q < P$, we have that $B_i^P \subseteq B_i^Q$ and hence $M_i(P) \subseteq M_i(Q)$. 

Suppose that $M_i$ has vertex $Q$. Then $M_i(P) = 0$ and $M_i(Q) \not= 0$, by Theorem \ref{brmain}. But then $M(P)$ is strictly contained in $M(Q)$, which is a contradiction.
\end{proof}

Let $\lambda$ be a partition of $n$, and let $Q$ be a $p$-subgroup of $S_n$. We consider the structure of $M^\lambda(Q)$. Observe that if $\lbrace t \rbrace$ is a $\lambda$-tabloid which is fixed by $Q$ and $\mathcal{O}$ is an orbit of $Q$ on $\lbrace 1,\ldots, n \rbrace$, then all elements of $\mathcal{O}$ must lie in the same row of $\lbrace t \rbrace$. Moreover, if $P$ is a $p$-subgroup of $S_n$ with the same orbits as $Q$, then $M^\lambda(P) = M^\lambda(Q)$. In particular, if $Q$ has $r_i$ orbits of length $p^i$, and if $\rho$ is the partition of $n$ with $r_i$ parts equal to $p^i$, then a Sylow $p$-subgroup of the Young subgroup $S_\rho$, say $P$, satisfies $M^\lambda(P) = M^\lambda(Q)$. It follows from Lemma \ref{goup} that the possible vertices of summands of $M^\lambda$ are Sylow $p$-subgroups of such Young subgroups $S_\rho$. 

Fix a partition $\rho$ with all its parts powers of $p$ and let $Q_\rho$ be a Sylow $p$-subgroup of $S_\rho$; in order to exploit the Brou\'e correspondence, we must understand the group $N_{S_n}(Q_\rho)$. Observe that, since $N_{S_n}(S_\rho)$ permutes orbits of $S_\rho$ of length $p^i$ as blocks for its action, $N_{S_n}(S_\rho)$ is conjugate to the direct product $(S_1 \wr S_{r_0}) \times (S_p \wr S_{r_1}) \cdots \times (S_{p^t} \wr S_{r_t})$. Consequently, $N_{S_n}(S_\rho)/S_\rho $ is isomorphic to  $S_{r_0} \times \cdots \times S_{r_t}$. On the other hand, applying the Frattini argument to $N_{S_n}(S_\rho)$, we have that $$N_{S_n}(S_\rho) = N_{N_{S_n}(S_\rho)}(Q_\rho) S_\rho \subset N_{S_n}(Q_\rho)S_\rho.$$
Since the right-hand side is contained in $N_{S_n}(S_\rho)$, we have that $N_{S_n}(S_\rho) \cong N_{S_n}(Q_\rho)S_\rho$. It follows from this and the Second Isomorphism Theorem that $$ N_{S_n}(S_\rho)/S_\rho = N_{S_n}(Q_\rho)/N_{S_\rho}(Q_\rho).$$
But the action of $S_\rho$ on $M^\lambda(Q_\rho)$ is trivial,so the structure of $M^\lambda(Q_\rho)$ as a module for $N_{S_n}(Q_\rho)/Q_\rho$ is the same as its structure considered as a module for $N_{S_n}(Q_\rho)/N_{S_\rho}(Q_\rho)$, which have already seen is isomorphic to $S_{r_0} \times \cdots \times S_{r_t}$. This justifies our considering $M^\lambda(Q_\rho)$ and $Y^\lambda(Q_\rho)$ as modules for this product of symmetric groups: it is simply more convenient to treat these Brauer quotients this way. We shall use this frequently without further comment throughout the paper.

\begin{lemma}\label{tens} \cite[Proposition 1]{Erd} There is an isomorphism of $\F N_{S_n}(Q_\rho)/Q_\rho$-modules: $$M^\lambda(Q_\rho) \cong \bigoplus_{\alpha \in T} M^{\alpha(0)} \otimes \cdots \otimes M^{\alpha(t)},$$ where $T$ is the set of all $t+1$-tuples $(\alpha_0,\ldots,\alpha_t)$ such that $\alpha(i)$ is a composition of $r_i$ and $\sum_{i=0}^t \alpha(i)p^i = \lambda$. \end{lemma}
\begin{proof}
Consider a $\lambda$-tabloid $\lbrace t \rbrace$ which is fixed by the action of $Q_\rho$, so the elements of each $Q_\rho$-orbit lie in the same row of $\lbrace t \rbrace$. For each $i$, let $\mathcal{O}^i_1,\ldots,\mathcal{O}^i_{r_i}$ denote the $Q_\rho$-orbits of length $p^i$. We define a composition $\alpha(i)$ of $r_i$ by setting the $j^\textrm{th}$ entry of $\alpha(i)$ equal to $k \in \N \cup \lbrace 0 \rbrace$ if the $j^\textrm{th}$ row of $\lbrace t \rbrace$ contains exactly $k$ elements from $\lbrace \mathcal{O}^i_1,\ldots,\mathcal{O}^i_{r_i} \rbrace$.

We now define a linear map $\phi :  M^\lambda(P) \rightarrow M^{\alpha(0)} \otimes \cdots \otimes M^{\alpha(t)}$ by setting $\phi(\lbrace t \rbrace) = \lbrace v_0 \rbrace \otimes \cdots \otimes \lbrace v_t \rbrace$, where $\lbrace v_i \rbrace$ is the $\alpha(i)$-tabloid which has entry $j$ in row $k$ if and only if the elements of $\mathcal{O}^i_j$ lie in row $k$ of $\lbrace t \rbrace$. This map induces a linear isomorphism between the two modules, as claimed, so all that remains is to show that $\phi$ is a $\F N_{S_n}(Q_\rho)/Q_\rho$-module isomorphism.

Let $g \in \F N_{S_n}(Q_\rho)/Q_\rho \cong S_{r_0} \times \cdots \times S_{r_t}$, so $g$ permutes the $P$-orbits of length $p^i$. Say that $\phi(g \lbrace t \rbrace) = \lbrace w_0 \rbrace \otimes \cdots \otimes \lbrace w_t \rbrace$. Suppose that the orbit $\mathcal{O}^i_j$ lies in row $k$ of $\lbrace t \rbrace$; then the orbit $\mathcal{O}^i_{g(j)}$ lies in row $k$ of $g \lbrace t \rbrace$ and hence the entry $g(j)$ is in row $k$ of $\lbrace w_i \rbrace$. But, by construction, the tabloid $\lbrace v_i \rbrace$ has entry $j$ in row $k$, whence $ g \lbrace v_i \rbrace$ has $g(j)$ in row $k$. Therefore, $g \lbrace v_i \rbrace = \lbrace w_i \rbrace$ for every $i$, and so $\phi$ is indeed a homomorphism, as required.  
\end{proof}

Finally, we shall require the following easy combinatorial result about partitions, often referred to as the $p$-adic expansion of a partition.

\begin{lemma}\label{exp}Let $n \in \N$ and $p$ be a prime. There is a bijection between the set of all partitions $\lambda$ of $n$ and all tuples of the form $(\alpha(0),\ldots,\alpha(t))$, such that $\alpha(i)$ is a $p$-restricted partition for each $i$, given by $ \lambda \leftrightarrow (\alpha(0),\ldots,\alpha(t))$ 
where $\lambda = \sum \alpha(i)p^i$.  \end{lemma}

We are now ready to complete the proof of Theorem \ref{main}. To do this, we must prove the following three assertions; this tripartite division is in the same spirit as the proof of the main theorem in \cite{Erd}.
\begin{enumerate}
\item Every summand of $M^\lambda$ is a Young module.
\item A vertex of $Y^\lambda$ is $Q_\rho$ (recall that this is a Sylow $p$-subgroup of the Young subgroup $S_\rho$).
\item $Y^\lambda(Q) \cong Y^{\lambda(0)} \otimes \cdots \otimes Y^{\lambda(s)},$ as $N_{S_n}(Q)/Q$-modules.
\end{enumerate}

\begin{proof}
We proceed by induction on $n$. If $n < p$, then $\F_p S_n$ is a semisimple algebra, so all its modules are projective. The number of indecomposable projective modules equals the number of simple $\F_p S_n$-modules, which is the number of partitions of $n$; this is the same as the number of $p$-restricted partitions of $n$.  Therefore, all summands of permutation modules $M^\mu$ are Young modules, giving (1). Furthermore, $Y^\lambda$ is projective, so has trivial vertex, whereas the Young subgroup $S_\rho$ has order coprime to $p$, so (2) holds. Since the vertex of $Y^\lambda$ is the identity group, $Y^\lambda$ is its own Brauer quotient, so (3) is trivially true.

Now suppose that $n \geq p$ and the result is true for all smaller degrees. The number of indecomposable projective $\F_p S_n$-modules equals the number of $p$-restricted partitions of $n$. We want to count the number of non-projective summands of $M^\mu$ as $\mu$ ranges over all partitions of $n$. By Theorem \ref{brcorr}, this is the same as the number of projective summands of all $M^\mu (Q_\rho)$, where $\rho$ ranges over all partitions of $n$ whose parts are all $p$-powers, excluding the partition $(1^n)$ (because we are excluding the trivial group as a vertex).

The Brauer quotient $M^\mu(Q_\rho)$ is a direct sum of modules of the form $M^{\alpha_0} \otimes \cdots \otimes M^{\alpha_t}$, by Lemma \ref{tens}. The indecomposable projective summands are therefore of the form $P^{\alpha(0)} \otimes \cdots  \otimes P^{\alpha(t)}$, where $P^{\alpha(i)}$ is an indecomposable projective module for $S_{r_i}.$ The number of possible $P^{\alpha(i)}$ is equal to the number of $p$-restricted partitions of $r_i$. Therefore, the total number of such summands is equal to the number of tuples of $p$-restricted partitions $(\alpha(0),\ldots,\alpha(t))$ such that $\sum \alpha(i) p^i$ is a partition of $n$, excluding the tuples just equal to $(\alpha(0))$. 

By Lemma \ref{exp}, the number of such summands is equal to the number of partitions of $n$, less the number of $p$-restricted partitions. Hence the total number of summands (projective and non-projective) equals the number of partitions of $n$. However, for each partition $\lambda$ of $n$, we already have the Young module $Y^\lambda$ as a summand of $M^\lambda$. Consequently, there can be no other summands, and (1) is established.  

We prove (2) and (3) by a further induction on the dominance order of partitions. Write $n = \sum a_ip^i$, the $p$-adic expansion of $n$. The module $Y^{(n)}$ is the trivial module, so it has vertex a Sylow $p$-subgroup of $S_n$. By the construction of Sylow $p$-subgroups of $S_n$ as iterated wreath products, it follows that $Y^{(n)}$ has vertex a Sylow $p$-subgroup of $S_\rho$, where $\rho$ has $a_i$ parts equal to $p^i$. Moreover, the Brauer quotient is the trivial module, which is isomorphic to $Y^{(a_0)} \otimes \cdots \otimes Y^{(a_t)}$. Now suppose that $\lambda < (n)$ and $\lambda$ is not $p$-restricted. Write $\lambda = \sum_{i=0}^t \lambda(i)p^i$, with each $\lambda(i)$ a $p$-restricted partition. Let $\rho$ and $Q_\rho$ be as in the statement of Theorem \ref{main}.

By Lemma \ref{tens}, $M^{\lambda(0)} \otimes \cdots \otimes M^{\lambda(t)}$ is a summand of $M^\lambda(Q_\rho)$. Since the degree of each $\lambda(i)$ is strictly smaller than that of $\lambda$, we may apply the inductive hypothesis to each tensor factor. Therefore, $M^\lambda(Q_\rho)$ has the indecomposable projective module $X:= Y^{\lambda(0)} \otimes \cdots \otimes Y^{\lambda(t)}$ as a direct summand. This corresponds to an indecomposable summand of $M^\lambda$ with vertex $Q_\rho$.   We have already seen that $M^\lambda$ is a direct sum of $Y^\lambda$ and other modules $Y^\mu$, where $\mu > \lambda$ (and, by (1), these are the only summands). By the inductive hypothesis, the Brauer quotient of $Y^\mu$ for $\mu > \lambda$ is not $X$, so $Y^\lambda$ has vertex $P$ and Brauer quotient $X$, as required. So (2) and (3) hold, except for the case when $\lambda$ is $p$-restricted.

Finally, we have seen that if $\lambda$ is not $p$-restricted, then $Y^\lambda$ has non-trivial vertex, so cannot be projective. It follows that all the remaining Young modules must be projective; in other words, if $\lambda$ is $p$-restricted, then $Y^\lambda$ is projective. 
\end{proof}

\section{Block Theory}
In this section, we set out the important tools we shall need from block theory in proving Nakayama's Conjecture. For a general introduction to block theory, we refer the reader to \cite[Chapter 4]{Alp}. At several points in the argument, we shall wish to pass between different types of module labelled by a partition of $n$. Our first lemma justifies this.

\begin{lemma}\label{chanlab}Let $\lambda$ be a $p$-regular partition of $n$ and $B$ be a block of $S_n$. The following are equivalent: \begin{enumerate}
\item $Y^\lambda$ lies in $B$;
\item Every summand of $S^\lambda$ lies in $B$;
\item $D^\lambda$ lies in $B$.
\end{enumerate} \end{lemma}
\begin{proof}
Assuming (1), let $e$ be the block idempotent corresponding to $B$; $Y^\lambda$ lies in $B$, so $e$ acts as the identity on $Y^\lambda$. Since $S^\lambda$ is a submodule of $Y^\lambda$, $e$ also acts as the identity on $S^\lambda$, whence every summand of $S^\lambda$ must lie in $B$. Similarly, $D^\lambda$ is a subquotient of $S^\lambda$, so if $e$ acts as the identity on $S^\lambda$, then $e$ acts as the identity on $D^\lambda$ as well.

Conversely, extending a composition series for $S^\lambda$ to one for $Y^\lambda$ shows that $D^\lambda$ is a composition factor of $Y^\lambda$. Young modules are indecomposable, so $Y^\lambda$ lies in the block $B$ if and only if $Y^\lambda$ has a composition factor in $B$. Therefore, if $D^\lambda$ lies in $B$, then $Y^\lambda$ also lies in $B$.
\end{proof}

Our next result gives a useful condition for a group to have only one block; to present it, we shall require the notion of covering, as defined in \cite[p.105]{Alp}. If $G$ is a group with normal subgroup $N$, and $B$ and $C$ are blocks of $G$ and $N$ respectively, we say that $B$ covers $C$ if there is some $\F G$-module $M$ lying in $B$ such that $M \Res_N$ has a summand lying in $C$. Recall that for a group $G$, the block in which the trivial $\F G$-module lies is called the principal block of $G$ and is denoted by $b_0(G)$.

\begin{lemma}\label{oneblock} Suppose that the group $G$ has a normal $p$-subgroup $L$ such that $C_G(L) \leq L$. Then $G$ has a unique block. \end{lemma}
\begin{proof}
We recall that $b_0(L)$ has defect group $L$. By \cite[Theorem 15.1(5)]{Alp}, there is a unique block $B$ of $G$ covering $b_0(L)$. Since $L$ is a $p$-group, $L$ has just one indecomposable projective module, and hence only one block. Moreover, every block of $L$ is covered by some block of $G$ by \cite[Theorem 15.1(4)]{Alp}, and so $G$ must have a unique block.
\end{proof}

We shall also need to understand how taking duals affects the block in which a $\F S_n$-module lies; the answer is provided by the following elementary lemma.

\begin{lemma}\label{dualblock}Let $M$ be a $\F S_n$-module lying in the block $B$ of $S_n$. Then the dual module $M^\star$ also lies in $B$.\end{lemma}
\begin{proof}
By considering each indecomposable summand of $M$ separately if necessary, we may assume that $M$ is indecomposable. Let $D^\lambda$ be a composition factor of $M$, so $D^\lambda$ lies in $B$. However, since $D^\lambda$ is self-dual, $D^\lambda$ is also a composition factor of $M^\star$: $M^\star$ is an indecomposable module with a composition factor in $B$, so $M^\star$ lies in $B$.
\end{proof}

The following generalisation of Brauer's Second Main Theorem will be essential in our argument; to state it, we recall the definition of the Brauer correspondence from \cite[p.101]{Alp}. If $H$ is a subgroup of $G$, and $C$ is a block of $H$, we say the block $B$ of $G$ \emph{corresponds} to $C$, and denote this by $C^G = B$, if $C$ (considered as a module for $H \times H$) is a direct summand of $(B \times B)\Res_{H \times H}$ and $B$ is the unique block of $G$ with this property.

\begin{lemma}\label{genbr2} \cite[Lemma 7.4]{Wild} Let $M$ be an indecomposable $p$-permutation $\F G$-module with vertex $P$, such that $M$ lies in the block $B$ of $G$. Let $Q$ be a subgroup of $P$, and suppose that the Brauer quotient $M(Q)$ has a summand in the block $C$ of $N_{G}(Q)$. Then $C^{G}$ is defined and $C^{G} = B$. \end{lemma}

\section{Proof of Nakayama's Conjecture}
We now begin the proof of Nakayama's Conjecture, which we break up into a number of steps. If $\gamma$ is a $p$-core, we denote by $b^\gamma$ the block of $S_n$ which is labelled by $\gamma$. The proof is by induction on $n$.

\subsection*{Step 1: Base Case}
If $n < p$, then every partition of $n$ is $p$-restricted, and so every Young module is projective by Theorem \ref{main}. Moreover, the algebra $\F S_n$ is semisimple, so each Young module is simple. It follows by \cite[Proposition 13.3(2)]{Alp} that any two Young modules lie in different blocks. On the other hand, all the partitions of $n$ are $p$-cores, so the result holds. Now suppose that $n \geq p$ and Nakayama's Conjecture holds for all symmetric groups of lower degree.

\subsection*{Step 2: Common Vertices} Let $\lambda$ be a partition of $n$ with $p$-adic expansion $\lambda = \sum_{i=0}^t \lambda(i)p^i$, and put $r_i = |\lambda(i)|$. Recall that by Theorem \ref{main} the tuple $(r_0,r_1,\ldots,r_t)$ determines the vertex of the module $Y^\lambda$.  We call the tuple $(r_0,r_1,\ldots,r_t)$ the \emph{$p$-type} of $\lambda$.

Let $\lambda$ and $\mu$ be partitions of $n$ of the same $p$-type which are not $p$-restricted (so $r_0 < n$). Write the $p$-adic expansions as $\lambda = \sum_{i=0}^t \lambda(i)p^i$, $\mu = \sum_{i=0}^s \mu(i)p^i$. Then $Y^\lambda$ and $Y^\mu$ have common vertex $Q$ as defined in Theorem \ref{main}, and their Brauer quotients satisfy \begin{align*}Y^\lambda(Q) \cong Y^{\lambda(0)} \boxtimes \cdots \boxtimes Y^{\lambda(t)}, \\ Y^\mu(Q) \cong Y^{\mu(0)} \boxtimes \cdots \boxtimes Y^{\mu(t)}, \end{align*} as $N_{S_n}(Q) \cong S_{r_{0}} \times N_{S_{n-{r_0}}}(Q)$-modules.~Since a $p$-core is necessarily $p$-restricted, $\lambda(0)$ has the same $p$-core as $\lambda$ and $\mu(0)$ has the same $p$-core as $\mu$. We may apply the inductive hypothesis to the first tensor factor, because $r_0 < n$. Moreover, the group $N_{S_{n-{r_0}}}(Q)$ has a unique block, by applying Lemma \ref{oneblock} with $L = Q$. Indeed, if $R$ is an elementary abelian subgroup of $Q$ generated by $p$-cycles, and $R$ has maximal rank among all subgroups of $Q$ with these properties, then $C_{N_{S_{n-{r_0}}}(Q)}(Q) \leq  C_{N_{S_{n-{r_0}}}(Q)}(R) = R \leq Q$.

Therefore, $Y^\lambda(Q)$ and $Y^\mu(Q)$ lie in the blocks $b^{c_p(\lambda)} \otimes b_0(N_{S_{n-{r_0}}}(Q))$ and $b^{c_p(\mu)} \otimes b_0(N_{S_{n-{r_0}}}(Q))$ of $N_{S_n}(Q)$, respectively. If $\lambda$ and $\mu$ have the same $p$-core, then these blocks are the same, and by Lemma \ref{genbr2}, $Y^\lambda$ and $Y^\mu$ lie in the same block. Conversely, suppose that $\lambda$ and $\mu$ have different $p$-core, but $Y^\lambda$ and $Y^\mu$ lie in the same block $B$ of $S_n$. Then, again by Lemma \ref{genbr2}, we have $(b^{c_p(\lambda)} \otimes b_0(N_{S_{n-{r_0}}}(Q)))^{S_n} = B$ and $(b^{c_p(\lambda)} \otimes b_0(N_{S_{n-{r_0}}}(Q)))^{S_n} = B$. However, $B$ has a unique Brauer correspondent with respect to $N_{S_n}(Q)$ by Brauer's First Main Theorem (see, for example \cite[Theorem 14.2]{Alp}, so this is a contradiction. We summarise our progress so far:

\begin{proposition} \label{step2} Let $\lambda$ and $\mu$ be partitions of $n$ of the same $p$-type which are not $p$-restricted. Then $Y^\lambda$ and $Y^\mu$ lie in the same block of $S_n$ if and only if $c_p(\lambda) = c_p(\mu)$. \end{proposition}

\subsection*{Step 3: Different $p$-type}We now aim to find a way to compare two Young modules which have different vertex. Given any possible $p$-type $(r_0,\ldots,r_t)$ with $t \not= 0$, we claim that there is a partition $\nu = (\nu_1,\nu_2,\ldots)$ of this $p$-type such that $\nu_1 - \nu_2 \geq p$. 

Indeed, since $t \not= 0$, there is some $i > 0$ such that $r_i \geq 1$. If $r_i = 1$, set $\nu(i) = (1)$, otherwise, we set $\nu(i) = (2, 1^{r_i-2})$; we observe that the partition $\nu(i)$ is $p$-restricted unless $p = 2$ and $r_i = 2$. We say that a 2-type is \emph{exceptional} if for every $i > 0$, $r_i$ is either 0 or 2. For now, suppose that our type is not exceptional, so at least one $\nu(j)$ is $p$-restricted, say $\nu(i)$ (where $i > 0$). Moreover, if $\nu$ is any partition of $p$-type $(r_0,\ldots,r_t)$ having $\nu(i)$ in its $p$-adic expansion, then $\nu_1 - \nu_2 \geq p^i(\nu(i)_1 - \nu(i)_2) \geq p$, as required.    

Let $\nu$ be such a partition and consider the Brauer quotient of $M^\nu$ with respect to the cyclic group $R := \langle (1,\ldots,p) \rangle$. Then, by Lemma \ref{tens}, $M^\nu(R)$ is isomorphic, as a $\F N_{S_n}(R)/R$-module, to a direct sum of modules of the form $M^\eta\boxtimes \F$, where $\eta$ is a composition of $n-p$ obtained by subtracting $p$ from a part of $\nu$. Define $\xi := (\nu_1-p,\nu_2,\ldots)$; by the previous paragraph, $\xi$ is a partition of $n-p$ and $c_p(\xi) = c_p(\nu)$.

Observe that if $\lambda \rhd \nu$, then the factor $M^\xi\boxtimes \F $ does not appear in the decomposition of $M^\lambda(R)$; if it did, then $\lambda$ could be obtained by adding $p$ to a part of $\xi$, but all such partitions are less than or equal to $\nu$ in the dominance order. Consequently, $Y^\xi \boxtimes \F$ does not appear in the decomposition of $Y^\lambda(R)$.  We have, by Lemma \ref{tens}, that $M^\nu(R) \supset M^\xi \boxtimes \F $. The module $M^\nu$ is a direct sum of $Y^\nu$ and modules $Y^\lambda$ for $\lambda \rhd \nu$. Since $Y^\xi \boxtimes \F$ appears in this decomposition, but not as a summand of any $Y^\lambda(R)$, it follows that $Y^\xi \boxtimes \F $ is a direct summand of $Y^\nu(R)$. 

By the inductive hypothesis, $Y^\xi$ lies in the block of $S_{n-p}$ labelled by $(c_p(\xi),w-1) = (c_p(\nu), w-1)$, where $w$ is the $p$-weight of $\lambda$. The group $N_{S_p}(R) \cong C_p \rtimes C_{p-1}$ has a unique block, by applying Lemma \ref{oneblock} with $L = R$. Hence, by induction, $Y^\nu(R)$ has a summand in the block $b^{c_p(\nu)} \otimes b_0(N_{S_p}(R))$ of $N_{S_n}(R)$. 

Consequently, suppose that $\lambda$ and $\mu$ are partitions of $n$ which are not $p$-restricted, and if $p = 2$, suppose further that neither $\lambda$ nor $\mu$ has exceptional 2-type. Then, by following the above procedure, we can find partitions $\nu_\lambda$ and $\nu_\mu$ of $n$ which are also not $p$-restricted such that: \begin{enumerate}
\item $\nu_\lambda$ has the same $p$-type as $\lambda$, and $\nu_\mu$ has the same $p$-type as $\mu$;
\item $c_p(\nu_\lambda) = c_p(\lambda)$ and $c_p(\nu_\mu) = c_p(\mu)$;
\item $Y^{\nu_\lambda}(R)$ has a summand in the block $b^{c_p(\lambda)} \otimes b_0(N_{S_p}(R))$ of $N_{S_n}(R)$ and $Y^{\nu_\mu}(R)$ has a summand in the block $b^{c_p(\mu)} \otimes b_0(N_{S_p}(R))$ of $N_{S_n}(R)$.
\end{enumerate}

By (1), (2) and Proposition \ref{step2}, $Y^\lambda$ lies in the same block as $Y^{\nu_\lambda}$, and $Y^\mu$ lies in the same block as $Y^{\nu_\mu}$. It then follows from (3) and Lemma \ref{genbr2} that $Y^\lambda$ and $Y^\mu$ lie in the same block of $S_n$ if $c_p(\lambda) = c_p(\mu)$. On the other hand, if $c_p(\lambda) \not= c_p(\mu)$, then $Y^{\nu_\lambda}(R)$ and $Y^{\nu_\mu}(R)$ have summands in different blocks, by the inductive hypothesis applied to $S_{n-p}$. By Lemma \ref{genbr2} and Brauer's First Main Theorem, $Y^{\nu_\lambda}$ and $Y^{\nu_\mu}$ lie in different blocks of $S_n$. It follows that $Y^\lambda$ and $Y^\mu$ also lie in different blocks of $S_n$.

We now come to the case of exceptional type: let $p=2$, $\lambda$ be a partition of exceptional 2-type, and let $Q$ be a vertex of $Y^\lambda$. Note that the support of $Q$ has size $n-r_0$. We define the partition $\hat{\lambda} = \lambda(0) + (n-r_0)$. By a similar argument to that given above for $Y^\nu(R)$, $Y^{\hat{\lambda}}(Q)$ has a summand in the block $b^{c_2(\lambda)} \otimes b_0(N_{S_{n-r_{0}}}(Q))$ of $N_{S_n}(Q)$. Hence, by Lemma~\ref{genbr2}, $Y^\lambda$ and $Y^{\hat{\lambda}}$ lie in the same block. Let $\sum_j \beta_j 2^j$ be the 2-adic expansion of $n-r_0$, where each $\beta_j \in \lbrace 0,1 \rbrace$, and note that some $\beta_j$ equals 1, because $n-r_0 \not= 0$. Then the 2-type of $\hat{\lambda}$ is $(r_0,\beta_1,\beta_2,\ldots)$, which is not an exceptional 2-type and therefore the above argument can be applied to $Y^{\hat{\lambda}}$. 

Indeed, if $\mu$ is another partition of $n$ which is not $p$-restricted and not of exceptional 2-type, the argument above shows that $Y^{\hat{\lambda}}$ and $Y^\mu$ lie in the same block if and only if $c_p(\hat{\lambda}) = c_p(\mu)$. But $Y^{\hat{\lambda}}$ and $Y^\lambda$ lie in the same block, and $c_p(\lambda) = c_p(\hat{\lambda})$, so we deduce that $Y^\lambda$ and $Y^\mu$ lie in the same block if and only if $c_p(\lambda) = c_p(\mu)$.    

\begin{proposition}\label{step3} Let $\lambda$ and $\mu$ be partitions of $n$, neither of which are $p$-restricted. Then $Y^\lambda$ and $Y^\mu$ lie in the same block if and only if $c_p(\lambda) = c_p(\mu)$. \end{proposition}

\subsection*{Step 4: Projective Case}
Now suppose that $\lambda$ is $p$-restricted. The module $Y^\lambda$ is projective and is its own Brauer quotient, so we require a different approach. We first deal with the case where $\lambda$ is a $p$-core.

\begin{proposition}\label{corep} Let $\gamma$ be a $p$-core. Then the $\F S_n$-Specht module $S^\gamma$ is simple and projective. \end{proposition}
\begin{proof}
We let $e_\gamma = \frac{\chi^\gamma(1)}{n!}\sum_{g \in S_n} \chi^\gamma(g^{-1}) g$; this is the primitive central idempotent over $\Q_p$ corresponding to $S^\gamma$. The group algebra $\Q_p S_n$ is isomorphic to a direct sum of matrix algebras, and $e_\gamma Q_p S_n \cong M_{\chi^\gamma(1)}(\Q_p)$. Clearly, $e_\gamma$ induces an algebra homomorphism $\theta : \Q_p S_n  \rightarrow M_{\chi^\gamma(1)}(\Q_p)$ with kernel $(1-e_\gamma) \Q_p S_n$. Denote by $\phi$ the restriction of $\theta$ to $\Z_p S_n$; clearly $\phi$ has kernel $(1-e_\gamma) \Z_p S_n$.

Now, we claim that the image of $\phi$ is $M_{\chi^\gamma(1)}(\Z_p)$; since the character table of $S_n$ is integral and $\gamma$ is a $p$-core, $e_\gamma$ is defined over $\Z_p$, so $\phi$ maps into $M_{\chi^\gamma(1)}(\Z_p)$. If $\beta \in \End_{\Z_p}(S^\gamma)$, then by \cite[p.162]{Webb} we have $$\beta = \frac{\chi^\gamma(1)}{n!} \sum_{g \in S_n} \Tr(\theta(g^{-1} \beta ) \theta(g), $$ whence $\phi$ is surjective.

We can therefore identify the $\Z_p S_n$-Specht module $S^\gamma$ with the space of row vectors, as a module for $M_{\chi^\gamma(1)}(\Z_p)$. Reducing this modulo $p \Z_p$, we obtain $S^\gamma / p\Z_p S^\gamma$, which is isomorphic to the $\F S_n$-Specht module $S^\gamma$, but is also isomorphic to the only simple module for $M_{\chi^\gamma(1)}(\F_p)$. It follows that $S^\gamma$ is projective and simple.
\end{proof}
A block containing a simple projective module is necessarily a block of defect zero by \cite[Corollary 6.3.4]{Benson}, and hence that simple projective module is in fact the only indecomposable module in the block up to isomorphism. So if $\gamma$ is a $p$-core and $\lambda$ is any other partition, then $Y^\lambda$ lies in the same block as $Y^\gamma$ if and only if $\lambda = \gamma$. 

All that remains is to establish the result for $p$-restricted partitions which are not $p$-cores. Recall that $S^{(1^n)}$ denotes the sign module for $S_n$.

\begin{definition} \label{mullmap} The Mullineux map $m$ is a bijective involution on the set of $p$-regular partitions of $n$, defined by $m(\eta) = \mu$ if and only if $D^\eta \otimes S^{(1^n)} \cong D^\mu$. \end{definition}

We recall the following result on the dual of a Specht module. 

\begin{theorem} \label{dualspecht} \cite[Theorem 8.15]{Jam}
The dual of the Specht module $S^\lambda$ is isomorphic to $S^{\lambda'} \otimes S^{(1^n)}$, where $\lambda'$ denotes the conjugate partition to $\lambda$.
\end{theorem}

If $\lambda$ is 2-restricted, then the conjugate partition $\lambda'$ is 2-regular and hence the Specht module $S^{\lambda'}$ is indecomposable by \cite[Theorem 3.2]{Wild}. In characteristic 2, the sign representation and the trivial representation coincide, so Theorem~\ref{dualspecht} and Lemma~\ref{dualblock} imply that $S^\lambda$ and $S^{\lambda'}$ lie in the same block. If both $\lambda$ and $\lambda'$ are 2-restricted, then $\lambda$ is a 2-core; otherwise we may apply our earlier arguments to the module $Y^{\lambda'}$. Therefore, in characteristic 2, we have already established the result.

For the rest of this step, suppose that $p$ is odd; recall that in odd characteristic, Specht modules are indecomposable. By Lemma \ref{dualblock}, $S^\lambda$ and $(S^\lambda)^\star \cong S^{\lambda'} \otimes S^{(1^n)}$ lie in the same block of $S_n$, say $B$. Now, $S^\lambda / \text{rad($S^\lambda$)} \otimes S^{(1^n)} \cong D^{\lambda'} \otimes S^{(1^n)} = D^{m(\lambda')}$ is a composition factor of $(S^\lambda)^\star$, and hence lies in $B$. Therefore, by Lemma \ref{chanlab}, $Y^\lambda$ and $Y^{m(\lambda')}$ lie in the same block. The key result is the following.

\begin{proposition}\label{multrans}Let $\lambda$ be a $p$-restricted partition of $n$. Then $\lambda \unlhd m(\lambda')$, with equality only if $\lambda$ is a $p$-core. \end{proposition}
\begin{proof}
Since $\lambda$ is $p$-restricted, $\lambda$ labels a simple $\F_p S_n$-module $D_\lambda$ which has projective cover $Y^\lambda$. There is an inclusion $D_\lambda \rightarrow S^\lambda$, and by duality we obtain a surjective restriction homomorphism $(S^\lambda)^\star \rightarrow (D_\lambda)^\star \cong D_\lambda$. Write $S^\lambda_\Z$ for the Specht module $S^\lambda$ defined over the $p$-adic integers $\Z_p$. The map $${(S^\lambda_\Z)}^\star \rightarrow {(S^\lambda_\Z / pS^\lambda_\Z)}^\star \cong (S^\lambda)^\star$$ is surjective, and composing this with the surjective restriction $(S^\lambda)^\star \rightarrow D_\lambda$ yields a surjective homomorphism ${(S^\lambda_\Z)}^\star \rightarrow D_\lambda$.  Let $P^\lambda_\Z$ denote the unique (up to isomorphism) indecomposable projective $\Z_p S_n$-module lifting $Y^\lambda$ to $\Z_p$, so there is a surjective homomorphism $P^\lambda_\Z \rightarrow D_\lambda$. Since $P^\lambda_\Z$ is projective, we obtain a non-zero homomorphism $P^\lambda_\Z \rightarrow {(S^\lambda_\Z)}^\star $. Let $P^\lambda_\Z$ have character  $\chi$; then the character $\chi^\lambda$ is a constituent of $\chi$, because the ordinary character of ${(S^\lambda_\Z)}^\star$ is $\chi^\lambda$.

Set $\mu = m(\lambda')$, and observe that $D_\lambda = D^\mu$. Hence $Y^\lambda$ is a projective cover of $D^\mu$, whence there is a surjective homomorphism $P^\lambda_\Z \rightarrow D^\mu$. There is also the canonical surjection from $S^\mu_\Z$ onto $D^\mu$, so again by the universal property of projective modules, there is a non-zero homomorphism $P^\lambda_\Z \rightarrow S^\mu_\Z$. Hence $\chi^{\mu}$ is also a constituent of $\chi$, so $\langle \chi, \chi^{\mu} \rangle \not= 0$. 

Now, we have $$\chi = \sum_{\nu} \langle \chi, \chi^\nu \rangle \chi^\nu  = \sum_{\nu} [S^\nu : D_\lambda]\chi^\nu,
$$ where the second equality follows from Brauer Reciprocity (\cite[Section 15.4]{Serre}). Hence $[S^\mu : D_\lambda] \not= 0$, and so $\mu \unrhd \lambda$. 

Finally, suppose that $\lambda = \mu$ (which makes sense because $\lambda$ is necessarily $p$-regular and $p$-restricted), so $D^\lambda = D_\lambda$: we claim then that $\lambda$ is a $p$-core. Therefore, $S^\lambda$ is simple, because otherwise $S^\lambda$ would have two composition factors isomorphic to $D^\lambda$, but $[S^\lambda : D^\lambda] = 1$.  Since both $\lambda$ and $\lambda'$ are $p$-regular, it follows from \cite[23.6(ii)]{Jam} that $p$ does not divide the product of the hook lengths in the Young diagram of $\lambda$. But then there are no rim hooks of length divisible by $p$ in the Young diagram of $\lambda$; by \cite[2.7.40]{JamKer}, $\lambda$ is a partition of $p$-weight zero, namely a $p$-core, as required.
\end{proof}

Define a map on $p$-restricted partitions by $f(\lambda) = m(\lambda')$. If $\lambda$ is $p$-restricted, but not a $p$-core, then repeatedly applying the above yields a chain of partitions $\lambda \lhd f(\lambda) \lhd f^2(\lambda) \lhd \ldots$ which are strictly increasing in the dominance order, and stops the first time we reach a partition which is not $p$-restricted. Say the chain terminates at $\mu$, so $\mu$ is not $p$-restricted; that is, $Y^\mu$ is non-projective. By a repeated application of the argument given in the paragraph preceding Proposition \ref{multrans}, the modules $Y^\lambda, Y^{f(\lambda)},\ldots,Y^\mu$ all lie in the same block. Since $\mu$ is not $p$-restricted, we can apply Proposition \ref{step3} to $Y^\mu$, and deduce the following:

\begin{proposition}\label{step4} Let $\lambda$ and $\mu$ be $p$-restricted partitions of $n$, neither of which are $p$-cores. Then $Y^\lambda$ and $Y^\mu$ lie in the same block if and only if $c_p(\lambda) = c_p(\mu)$. \end{proposition}

This, combined with Propositions \ref{step3} and \ref{corep}, completes the proof of Nakayama's Conjecture.

\subsection*{Acknowledgements} I am grateful to my supervisor Dr. Mark Wildon for his many insightful comments and valuable suggestions. I would also like to thank Jasdeep Kochhar for identifying an error in an earlier draft of this paper.


\begin{thebibliography}{99}

\bibitem{Alp}
{\sc Alperin, J.~L.}
\newblock {\em Local representation theory}, {Cambridge studies in advanced mathematics, vol. 11},
\newblock Cambridge University Press, 1986.

\bibitem{Benson}
{\sc Benson, D. J.}
\newblock {\em Representations and cohomology I},
{Cambridge studies in advanced mathematics, vol. 30,}
\newblock Cambridge University Press, 1995.

\bibitem{Brauer}
{\sc Brauer, R.}
\newblock {\em On a conjecture by Nakayama},
\newblock Trans. Royal Soc. Canada. III, \textbf{41} (1947), 11--19.

\bibitem{Br1}
{\sc Brou{\'{e}}, M.}
\newblock {\em On {S}cott modules and {$p$}-permutation modules: an approach through
  the {B}rauer homomorphism},
\newblock  Proc.~Amer.~Math.~Soc.,\/ \textbf{93} no. 3 (1985), 401--408.

\bibitem{Br2}
{\sc Brou{\'{e}}, M.}
\newblock {\em Les $l$-blocs des groupes $\GL(n, q)$ et $U(n, q^2)$ et leurs structures locales},
\newblock Ast\'erisque 133-–134 (1986), 159-–188.

\bibitem{Erd}
{\sc Erdmann, K.}
\newblock {\em Young modules for symmetric groups},
\newblock { J.~Aust.~Math.~Soc.,\/} 71 no. 2 (2001), 201--210.

\bibitem{ErdSch}
{\sc Erdmann, K. and Schroll, S.}
\newblock {\em On Young modules of general linear groups},
\newblock { J. Alg.,\/} \textbf{310} no. 1, (2007), 434--451.

\bibitem{Grab}
{\sc Grabmeier, J.}
 \newblock {\em Unzerlegbare Moduln mit trivialer Younquelle und Darstellungstheorie der Schuralgebra},
 \newblock	{ Bayreuth. Math. Schr.,} \textbf{20} (1985), 9--152.

\bibitem{Jam}
{\sc James, G.~D.}
\newblock {\em The representation theory of the symmetric groups}, vol.~682 of
  { Lecture Notes in Mathematics}.
\newblock Springer-Verlag, 1978.

\bibitem{JamKer}
{\sc James, G.~D. and Kerber, A.}
\newblock {\em The representation theory of the symmetric group},
\newblock Addison-Wesley, 1981.

\bibitem{Kly}
{\sc Klyachko, A. A.}.
\newblock {\em Direct Summands of Permutation Modules},
\newblock {Selecta Math. Soviet. \/} I(3), (1983--1984), 45--55.

\bibitem{MeiTap}
{\sc Meier, N. and Tappe, J.}
\newblock {\em Ein neuer Beweis der Nakayama-Vermutung \"uber die Blockstruktur symmetrischer Gruppen},
\newblock Bull. Lond. Math. Soc., \textbf{8} (1976), no. 1, 34--37.

\bibitem{Mur}
{\sc Murphy, G. E.}.
\newblock {\em The idempotents of the symmetric group and Nakayama's conjecture},
\newblock {J. Alg.\/}, \textbf{81} (1983), 258--265.

\bibitem{Robinson}
{\sc Robinson, G. de B.}
\newblock {\em On a conjecture by Nakayama},
\newblock Trans. Royal Soc. Canada. III, \textbf{41} (1947), 20--25.

\bibitem{Serre}
{\sc Serre, J.~P.}
\newblock {\em Linear representations of finite groups},
 GTM, vol.~42, Springer, 1977.
 
\bibitem{Webb}
{\sc Webb, P.}
\newblock {\em A Course in Finite Group Representation Theory}, Cambridge University Press, 2016.


\bibitem{Wild}
{\sc Wildon, M}.
\newblock {\em Vertices of Specht modules and blocks of the symmetric group},
\newblock {J. Alg.\/}, \textbf{323} (2010), 2242--2256.



\end{thebibliography}
 \end{document}